\documentclass{article}
\usepackage[utf8]{inputenc}
\pdfoutput=1

\usepackage{amssymb,amsmath, amsthm, amsfonts,pgfplots}
\pgfplotsset{compat=1.17}
\usepackage{mathrsfs}
\usepackage[ruled,linesnumbered]{algorithm2e}
\usepackage{comment}

\usepackage[margin=1in]{geometry}
\usepackage{graphicx, subcaption}
\usepackage{xcolor}
\usepackage{tikz-cd}
\usepackage{biblatex}
\def\spann{\mathop{\rm span}\nolimits}

\newcommand{\bbZ}{\mathbb{Z}}

\newcommand{\bbR}{\mathbb{R}}

\newcommand{\cD}{\mathcal{D}}

\newcommand{\cL}{\mathcal{L}}

\newcommand{\cV}{\mathcal{V}}

\newcommand{\cR}{\mathcal{R}}

\newtheorem{theorem}{Theorem}[section]
\newtheorem{corollary}{Corollary}[theorem]

\newtheorem{proposition}[theorem]{Proposition}

\theoremstyle{definition}
\newtheorem{example}{Example}[section]

\usepackage{hyperref}
\hypersetup{
    colorlinks=true,
    linkcolor = blue,
    citecolor = blue,
    urlcolor = blue
}

\title{Alternating Minimization for Regression with Tropical Rational Functions}
\author{Alex Dunbar\thanks{Department of Mathematics, Emory University, Atlanta, GA, USA} \and Lars Ruthotto$^*$}
\date{}

\begin{document}

\maketitle

\begin{abstract}
    We propose an alternating minimization heuristic for regression over the space of tropical rational functions with fixed exponents. The method alternates between fitting the numerator and denominator terms via tropical polynomial regression, which is known to admit a closed form solution. We demonstrate the behavior of the alternating minimization method experimentally. Experiments demonstrate that the heuristic provides a reasonable approximation of the input data. Our work is motivated by applications to ReLU neural networks, a popular class of network architectures in the machine learning community which are closely related to tropical rational functions.
\end{abstract}

\section{Introduction}

Tropical algebra uses a semiring structure on $\bbR \cup \{-\infty\}$ where the tropical sum of two elements is their maximum and tropical multiplication is standard addition. In this setting, $n$-variable tropical polynomials  are functions that are the pointwise maximum of finitely many affine functions with slopes in a finite set $W \subseteq \bbZ^n_{\geq 0}$. Such functions are piecewise linear and convex. Tropical rational functions are the standard difference between two tropical polynomials and therefore continuous piecewise linear functions. 

In this paper, we are interested in fitting tropical rational functions to data and developing a numerical method for solving regression problems for this function class. Specifically, we consider the $\ell^\infty$ regression problem over tropical rational functions with exponents in a fixed finite set $W \subseteq \bbZ^n_{\geq 0}$: Given a dataset $\cD = \left\{(\mathbf{x}^{(1)},y^{(1)}),\allowbreak  (\mathbf{x}^{(2)},y^{(2)}),  \ldots,  (\mathbf{x}^{(N)},y^{(N)}) \right\}\subseteq \bbR^n\times \bbR$, find
\begin{equation}\label{eq:ProblemDef}
\arg\min_{\mathbf{p},\mathbf{q}}\left\|\begin{bmatrix}\max_{\mathbf{w} \in W}(\mathbf{w}^\top \mathbf{x}^{(1)} + p_\mathbf{w}) \\  \max_{\mathbf{w} \in W}(\mathbf{w}^\top \mathbf{x}^{(2)} + p_\mathbf{w})\\ \vdots \\ \max_{\mathbf{w} \in W}(\mathbf{w}^\top \mathbf{x}^{(N)} + p_\mathbf{w})\end{bmatrix} - \begin{bmatrix}\max_{\mathbf{w} \in W}(\mathbf{w}^\top \mathbf{x}^{(1)} + q_\mathbf{w}) \\ \max_{\mathbf{w} \in W}(\mathbf{w}^\top \mathbf{x}^{(2)} + q_\mathbf{w}) \\ \vdots \\ \max_{\mathbf{w} \in W}(\mathbf{w}^\top \mathbf{x}^{(N)} + q_\mathbf{w})\end{bmatrix} - \begin{bmatrix} y^{(1)}\\ y^{(2)}\\ \vdots \\ y^{(N)} \end{bmatrix}\right\|_\infty,
\end{equation}
where the coefficient vectors $\mathbf{p} = (p_{\mathbf{w}})_{\mathbf{w} \in W}$ and $\mathbf{q} = (q_{\mathbf{w}})_{\mathbf{w} \in W}$ define the tropical polynomials $p(\mathbf{x}) = \max_{\mathbf{w} \in W}(\mathbf{w}^\top \mathbf{x} + p_{\mathbf{w}})$ and $q(\mathbf{x}) = \max_{\mathbf{w} \in W}(\mathbf{w}^\top \mathbf{x} + q_{\mathbf{w}})$, respectively. 
Since optimizing both vectors simultaneously is difficult, we propose a heuristic for the solution of \eqref{eq:ProblemDef} that alternates between solving the tropical polynomial regression problem of finding the optimal vector $\mathbf{p}$ for fixed $\mathbf{q}$ and the similar problem of finding the optimal $\mathbf{q}$ given fixed $\mathbf{p}$.

Our proposed heuristic alternates between updating $\mathbf{p}$ and $\mathbf{q}$ using results from tropical polynomial regression; see, e.g., \cite{AKIAN2011ApproxSemiModules,maragos2020multivariate,maragos2019tropical}. In each substep, we leverage the algebraic structure of tropical polynomials and use the fact that the closed-form solution involves only (min-plus and max-plus) matrix-vector products and vector addition. This renders each iteration of our heuristic  computationally cheap. Geometrically, our proposed heuristic searches the nondifferentiability locus of the $\ell^\infty$ loss in a way such that the loss is nonincreasing. In our experiments, a few iterations of this heuristic provide a reasonable approximation of the input data.

Problem \eqref{eq:ProblemDef} fits into the framework of piecewise linear regression. Such problems have received some attention from the optimization community \cite{kazda_nonconvex_2021,magnani_convex_2009,toriello_fitting_2012} and have recently seen great interest from the deep learning community. In this setting, piecewise linear functions are commonly parametrized using \emph{ReLU neural networks}, functions which are expressible as the repeated compositions of affine transformations with a ReLU activation function $\sigma(\mathbf{x}) = \max(\mathbf{x},0)$; see, e.g., \cite{arora2018understanding,daubechies_nonlinear_2022,nair2010rectified}.  Such functions have proven to be very expressive, and the optimization problem over ReLU networks, although being plagued by non-convexity and non-smoothness, can often be solved to a reasonable accuracy with  variants of stochastic gradient descent.

Recent work \cite{Charisopoulos18TropicalApproach,Maragos2021TropML,zhang2018tropical} has shown that ReLU neural networks correspond to tropical algebraic objects. This connection has been leveraged to analyze the complexity of a neural network by counting its linear regions \cite{Charisopoulos18TropicalApproach,zhang2018tropical}, minimize trained networks \cite{smyrnis2020multiclass,Smyrnis2019TropicalDivision}, and extract linear regions of a trained network \cite{trimmel2021tropex}.

Piecewise linear regression utilizing a parametrization through max-plus algebra is studied in \cite{kazda_nonconvex_2021,toriello_fitting_2012}, where the optimization problem is interpreted through mixed integer programming. In these works, the authors minimize the $\ell^2$ norm and allow the set $W$ to vary in $\bbR^n$ during the optimization. Our approach differs in that we fix $W$ and use the $\ell^\infty$ norm as an objective function, allowing the heuristic to utilize the algebraic structure of tropical polynomials.

The remainder of the paper is organized as follows: Section \ref{sec:Background} reviews the relevant background from tropical algebra and ReLU neural networks. Section \ref{sec:Algorithm} presents the alternating algorithm for tropical regression. Section \ref{sec:Expirements} details numerical experiments with tropical rational regression. Finally, Section \ref{sec:Conclusions} presents concluding remarks and directions for future work. 

\section{Background} \label{sec:Background}

In this section, we review relevant background from tropical algebra, tropical polynomial regression, and ReLU neural networks. We adopt the following notational conventions: Bold lowercase letters denote vectors and bold uppercase letters denote matrices. If $\mathbf{x}$ is a vector, $x_i$ is the $i^{\text{th}}$ component of $\mathbf{x}$. Collections of vectors are indexed by superscripts in parentheses. Sequences are denoted by superscripts without parentheses. The all ones vector is denoted $\mathbf{1}$. 

\subsection{Tropical Algebra}

This section briefly recalls relevant ideas and notation from tropical algebra. A more thorough introduction can be found in \cite{brugallé2015TropicalIntro} and a standard reference is \cite{MaclaganSturmfels}. Tropical geometry has recently seen applications outside of algebraic geometry in optimization and statistics \cite{gartner2008tropicalSVM,joswig_monomial_2020,pachter_tropical_2004,tang2020tropical,Yoshida2021TropicalSVMeval}. The survey article \cite{Maragos2021TropML} provides an overview of applications of tropical geometry in machine learning. 

The main object of study in tropical algebra is the \emph{tropical semiring} $\mathbb{T}:=(\bbR \cup \{-\infty\},\oplus, \odot)$. Tropical addition is $a \oplus b = \max(a,b)$ and tropical multiplication is $a \odot b = a + b$. Tropical addition and tropical multiplication are both associative and commutative. The multipicative identity is $0$ and every finite element $a$ has a tropical multiplicative inverse $-a$. The additive identity is $-\infty$ and no element of $\mathbb{T}$ has an additive inverse. Tropical exponentials are repeated tropical multiplication and denoted $a^{\odot w}:= wa$ for $w \in \mathbb{Z}$. 

Given a collection of $n$ variables $x_1,x_2,\ldots, x_n$, we use multi-index notation to describe the \emph{tropical monomial} 

\[c \odot \mathbf{x}^{\odot \mathbf{w}} := c + \mathbf{w}^{\top}\mathbf{x} = c + w_1x_1 + w_2x_2 + \ldots + w_n x_n .\]
 Analogously to standard algebra, the \emph{tropical polynomials} in $n$ variables are defined as finite sums of tropical monomials $\mathbf{x}^{\odot \mathbf{w}}$. That is,  
\[\mathbb{T}[x_1,x_2,\ldots, x_n] := \left\{\bigoplus_{{\bf w} \in W} c_{{\bf w}} \odot \mathbf{x}^{\odot {\bf w}}\middle\vert c_{\mathbf{w}} \in \mathbb{T}, W \subseteq \mathbb{Z}^n_{\geq 0}\text{ finite}\right\} = \left\{\max_{\mathbf{w}\in W}(c_{\mathbf{w}}+{\mathbf{w}}^\top \mathbf{x})\middle\vert c_{{\bf w}} \in \mathbb{T}, W \subseteq \mathbb{Z}^n_{\geq 0}\text{ finite}\right\}.\]
The set of tropical polynomials is also a semiring with operations extending tropical addition and tropical multiplication. Tropical polynomials are convex piecewise linear functions. 

Finally, \emph{tropical rational functions} are functions which are tropical quotients of tropical polynomials, 

\[\mathbb{T}(x_1,x_2,\ldots, x_n):= \{p(\mathbf{x}) \oslash q(\mathbf{x}) \;\vert\; p,q \in \mathbb{T}[x_1,x_2,\ldots, x_n]\} = \{p(\mathbf{x}) - q(\mathbf{x}) \;\vert\; p,q \in \mathbb{T}[x_1,x_2,\ldots, x_n]\}.\]%
 Tropical rational functions are piecewise linear but not necessarily convex. However, they are the difference of convex functions. For a fixed set $W \subseteq \bbZ^n_{\geq 0}$, we use $\mathbb{T}[\mathbf{x}]_W$ to denote the tropical polynomials with exponents in $W$. Similarly, $\mathbb{T}(\mathbf{x})_{W}$ denotes the tropical rational functions with exponents in $W$. When $W = \{\mathbf{w} \in \bbZ^n \vert\,  0 \leq w_i \leq d \text{ for } i = 1,2,\ldots, n\}$, we say that functions $f \in \mathbb{T}(\mathbf{x})_W$ have \emph{degree} $d$.  

\paragraph{Tropical Linear Algebra} Many concepts from classical linear algebra over a field generalize to $\mathbb{T}^n$. Given $\mathbf{u},\mathbf{v}\in \mathbb{T}^n$, define vector addition as the componentwise maximum 

\[\mathbf{u}\oplus\mathbf{v} = \max(\mathbf{u},\mathbf{v}) := \begin{bmatrix} \max(u_1,v_1)& \max(u_2,v_2)& \cdots &\max(u_n,v_n) \end{bmatrix}^\top.\]
 For $\lambda \in \mathbb{T}$ and $\mathbf{u} \in \mathbb{T}^n$, define scalar multiplication as

\[\lambda\odot\mathbf{u} := \lambda \mathbf{1} + \mathbf{u} = \begin{bmatrix}u_1 + \lambda & u_2 + \lambda & \cdots  & u_n + \lambda \end{bmatrix}^\top.\]
 As in classical linear algebra, maps $\mathbb{T}^n \to \mathbb{T}^m$ which are compatible with the vector addition and scalar multiplication on $\mathbb{T}^n$ can be represented by matrix multiplication \cite{AKIAN2011ApproxSemiModules}. Given an $m\times n$ matrix $\mathbf{A}$ and a vector $\mathbf{u} \in \mathbb{T}^n$, \emph{max-plus matrix-vector multiplication} is defined as 

\[\mathbf{A} \boxplus \mathbf{x} := \begin{bmatrix} a_{1,1} & a_{1,2} & \ldots & a_{1,n}\\ a_{2,1} & a_{2,2} & \ldots & a_{2,n}\\ \vdots & \vdots & \vdots & \vdots \\ a_{m,1} & a_{m,2} & \ldots & a_{m,n}\\\end{bmatrix} \boxplus \begin{bmatrix} u_1\\u_2\\\vdots\\ u_n \end{bmatrix} = \begin{bmatrix} \max_{1\leq j\leq n}(a_{1,j} + u_j)\\ \max_{1\leq j\leq n}(a_{2,j} + u_j)\\ \vdots \\ \max_{1\leq j\leq n}(a_{m,j} + u_j) \end{bmatrix}. \]

We also work with the dual \emph{min-plus matrix-vector multiplication}

\[\mathbf{A} \boxplus' \mathbf{x} := \begin{bmatrix} a_{1,1} & a_{1,2} & \ldots & a_{1,n}\\ a_{2,1} & a_{2,2} & \ldots & a_{2,n}\\ \vdots & \vdots & \vdots & \vdots \\ a_{m,1} & a_{m,2} & \ldots & a_{m,n}\\\end{bmatrix} \boxplus' \begin{bmatrix} u_1\\u_2\\\vdots\\ u_n \end{bmatrix} = \begin{bmatrix} \min_{1\leq j\leq n}(a_{1,j} + u_j)\\ \min_{1\leq j\leq n}(a_{2,j} + u_j)\\ \vdots \\ \min_{1\leq j\leq n}(a_{m,j} + u_j) \end{bmatrix}. \]

\paragraph{Tropical Hypersurfaces} In classical algebraic geometry, the zero set of a polynomial is called a \emph{hypersurface}. The tropical analog for a tropical polynomial $p \in \mathbb{T}[x_1,x_2,\ldots, x_n]$ given by $p(\mathbf{x}) = \max_{\alpha \in W}(\alpha^\top \mathbf{x} + p_\alpha)$ is the \emph{tropical hypersurface}  

\[\begin{aligned}
\cV(p) &:= \{\mathbf{x} \in \bbR^n \;\vert\; p(\mathbf{x}) = \mathbf{w}^\top \mathbf{x} + p_{\mathbf{w}} = \mathbf{v}^\top \mathbf{x} + p_{\mathbf{v}} \text{ for some } \mathbf{w} \not = \mathbf{v} \in W\}\\
&= \{\mathbf{x} \in \bbR^n \;\vert\; p \text{ is not differentiable at } \mathbf{x}.\}
\end{aligned}\]

 Tropical hypersurfaces can be given the structure of a polyhedral complex and the connected components of the set $\bbR^n \setminus \cV(p)$ are open polyhedra. A well-known result about tropical hypersurfaces is that they are determined by the polyhedral geometry of their coefficients; see, e.g., \cite[Proposition 3.1.6]{MaclaganSturmfels}. 

 If $f = p \,\oslash\, q$ is a tropical rational function, then the nondifferentiability locus of $f$ is contained in $\cV(p) \cup \cV(q)$ and this containment can be proper.

 \begin{example}\label{ex:Nondiff_trop}
Consider the tropical polynomials $p = 0\oplus x_1 \oplus x_2$ and $q = x_1 \oplus x_2$ and the tropical rational function $f = p-q = 0 \oplus x_1 \oplus x_2 - x_1 \oplus x_2$. Now, 

\[\cV(p) = \{(x_1,x_2)\in \bbR^2\vert x_1 = x_2 \geq 0 \text{ or } x_1 = 0 \geq x_2 \text{ or } x_2 = 0 \geq x_1\}, \quad \cV(q) = \{(x_1,x_2) \in \bbR^2 \vert x_1 = x_2\},\]
 and the nondifferentiability locus of $f$ is 

\[X = \{(x_1,x_2) \in \bbR^2 \vert x_1 = x_2 \leq 0\}\cup \{(x_1,x_2)\in \bbR^2 \vert x_1 = 0, x_2\leq 0\}\cup \{(x_1,x_2) \in \bbR^2 \vert x_2 = 0, x_1 \leq 0\}.\]

These sets are shown in Figure \ref{fig:Nondiff_trop_fig}.

\begin{figure}[t]
\begin{subfigure}[b]{0.3\textwidth}
\centering
\includegraphics{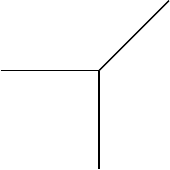}
\caption{$\cV(p)$ }
\end{subfigure}
\hfill
\begin{subfigure}[b]{0.3\textwidth}
\centering
\includegraphics{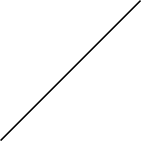}
\caption{$\cV(q)$}
\end{subfigure}
\hfill
\begin{subfigure}[b]{0.3\textwidth}
\centering
\includegraphics{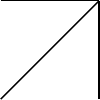}
\caption{$X$}
\end{subfigure}
\caption{The tropical hypersurfaces $\cV(p)$ and $\cV(q)$ and the nondifferentiability locus $X$ of $f = p-q$ in Example \ref{ex:Nondiff_trop}. The tropical hypersurfaces $\cV(p)$ and $\cV(q)$ divide $\bbR^2$ into polyhedral regions while $X$ divides $\bbR^2$ into regions which can be described as a finite union of polyhedra.}
\label{fig:Nondiff_trop_fig}
\end{figure}

\end{example}

\subsection{Weighted Lattices and Tropical Polynomial Regression}\label{sec:Lattice_Regression}

The vector addition and scalar multiplication on $\mathbb{T}^n$ are compatible with the partial order on $\mathbb{T}^n$ given by componentwise comparison. In \cite{maragos2019tropical}, this compatibility is studied from an optimization perspective in the framework of \emph{weighted lattices}, algebraic structures where operations are compatible with partial orders. Similar structures are discussed in detail in \cite{cuninghame-green_minimax_1979}. In \cite{HOOK2019maxplus2norm,maragos2020multivariate, maragos2019tropical,tsiamis2019sparsity,tsilivis2022towardSparsityWeighted}, this framework is leveraged to find optimal subsolutions to tropical-linear systems of equations and in particular solve tropical polynomial regression problems. We summarize the key points in this section, closely following the presentation in \cite{maragos2019tropical}. 

The vector addition on $\mathbb{T}^n$ is idempotent and gives a partial order where $\mathbf{u}\leq \mathbf{v}$ if and only if $\max(\mathbf{u},\mathbf{v}) = \mathbf{v}$. Similarly, the componentwise minimum gives the dual lattice structure. If $h:\mathbb{T}^n \to \mathbb{T}^m$ and $g: \mathbb{T}^m \to \mathbb{T}^n$ are two functions, then $h$ is a \emph{dilation} if $h(\max(\mathbf{u},\mathbf{v})) = \max(h(\mathbf{u}),h(\mathbf{v}))$, the function $g$ is an \emph{erosion} if $g(\min(\mathbf{u},\mathbf{v})) = \min(g(\mathbf{u}),g(\mathbf{v}))$, and the pair $(h,g)$ is an \emph{adjunction} if $h(\mathbf{u}) \leq \mathbf{v}$ is equivalent to $\mathbf{u} \leq g(\mathbf{v})$. As an example, for an $m \times n$ matrix $\mathbf{A}$, the map $\mathbb{T}^n \to \mathbb{T}^m$ given by $\mathbf{A} \boxplus \mathbf{u}$ is a dilation and the map $\mathbb{T}^m \to \mathbb{T}^n$ given by $(-\mathbf{A}^\top)\boxplus' \mathbf{u}$ is an erosion.

\begin{theorem}[\cite{maragos2019tropical}]\label{thm:dilation_erosion}
Given a dilation $h:\mathbb{T}^n \to \mathbb{T}^m$, there is a unique erosion $g:\mathbb{T}^m \to \mathbb{T}^n$ given by 

\[g(\mathbf{v}) = \max\{\mathbf{u}\in \mathbb{T}^n \vert h(\mathbf{u}) \leq \mathbf{v}\}\]
 such that $(h,g)$ is an adjunction. 
\end{theorem}

Applying Theorem \ref{thm:dilation_erosion} to max-plus matrix multiplication map $\mathbf{u}\mapsto \mathbf{A} \boxplus \mathbf{u}$ gives that the unique erosion to make an adjunct pair is the map $\mathbf{v} \mapsto (-\mathbf{A}^\top)\boxplus'\mathbf{v}$. This follows because $\max_{1\leq j\leq n}(u_j + a_{ij}) \leq v_i$ for all $1\leq i\leq m$ if and only if $u_j + a_{ij}\leq v_i$ for all $i,j$, which happens if and only if $u_j \leq \min_{1\leq i\leq m}(v_i - a_{ij})$ for all $j$. So, if $\mathbf{A}\boxplus \mathbf{u} \leq \mathbf{v}$, then $\mathbf{u}\leq (-\mathbf{A}^\top )\boxplus'\mathbf{v}$.

\begin{theorem}[\cite{cuninghame-green_minimax_1979}]\label{thm:opt_sol}

Let $\mathbf{A} \in \mathbb{T}^{m\times n}$ and $\mathbf{b} \in \mathbb{T}^m$.

\begin{itemize}
\item For any $\ell = 1,2,3,\ldots$, the optimal solution to

\begin{equation}\label{eq:subsol_approx}
\arg\min_{\mathbf{u}}\|\mathbf{A}\boxplus \mathbf{u} - \mathbf{b}\|_\ell \quad \mathrm{s.t.} \quad \mathbf{A}\boxplus \mathbf{u} \leq \mathbf{b}
\end{equation}
 is $\hat{\mathbf{u}} = (-\mathbf{A}^\top)\boxplus'\mathbf{b}$.

\item The optimal solution to 

\begin{equation}\label{eq:inf_norm_sol}
\arg\min_{\mathbf{u}}\|\mathbf{A}\boxplus \mathbf{u} - \mathbf{b}\|_\infty 
\end{equation}
 is $\hat{\mathbf{u}} + \frac{1}{2}\|\mathbf{A}\boxplus \hat{\mathbf{u}} - \mathbf{b}\|_\infty$, where $\hat{\mathbf{u}}$ is defined as in the previous part. 

\end{itemize}

\end{theorem}

We sketch a proof of the first part of Theorem \ref{thm:opt_sol} to demonstrate how the algebraic structure on $\mathbb{T}^n$ interacts with optimization. In particular, we use the algebraic structure to demonstrate why the solution to \eqref{eq:subsol_approx} is independent of $\ell$. 

\begin{proof}
Note that because max-plus matrix-vector multiplication is a dilation, we have that if $\mathbf{u},\mathbf{v} \in \mathbb{T}^n$ have $\mathbf{u} \leq \mathbf{v}$ then $\mathbf{A} \boxplus \mathbf{u} \leq \mathbf{A}\boxplus \mathbf{v}$. Now, if $\mathbf{u}$ is feasible to \eqref{eq:subsol_approx}, then $\mathbf{u} \leq (-\mathbf{A})^\top \boxplus' \mathbf{b} = \hat{\mathbf{u}}$. This implies that $\mathbf{A} \boxplus \mathbf{u} \leq \mathbf{A} \boxplus \hat{\mathbf{u}} \leq \mathbf{b}$ and therefore for each component $i = 1,2,\ldots, m$, it must be the case that $\left(\mathbf{b} - \mathbf{A}\boxplus \mathbf{u}\right)_i\geq \left(\mathbf{b} - \mathbf{A}\boxplus \hat{\mathbf{u}}\right)_i$. So, for any finite $\ell$ and any feasible $\mathbf{u}$, $\|\mathbf{A}\boxplus \mathbf{u} - \mathbf{b}\|_\ell \geq \|\mathbf{A}\boxplus \hat{\mathbf{u}} - \mathbf{b}\|_\ell$.
\end{proof}

Theorem \ref{thm:opt_sol} allows us to solve the tropical polynomial regression problem. Given data points $(\mathbf{x}^{(1)},y^{(1)}),\allowbreak  (\mathbf{x}^{(2)},y^{(2)}),  \ldots,  (\mathbf{x}^{(N)},y^{(N)}) \in \bbR^n\times \mathbb{R}$ and a finite subset $W \subseteq \mathbb{Z}^n_{\geq 0}$, set $\mathbf{X}$ to be the $n \times |W|$ matrix with $(i,\mathbf{w})$ entry $X_{i,\mathbf{w}} = \mathbf{w}^\top \mathbf{x}^{(i)}$. This is the tropical analog of a Vandermonde matrix. Then, if $\mathbf{y}$ is the vector $\begin{bmatrix}y^{(1)},y^{(2)},\ldots, y^{(N)} \end{bmatrix}^\top$,  the tropical polynomial $p$ which minimizes $\max_{1\leq i\leq N}|p(\mathbf{x}^{(i)})-y^{(i)}|$ is 

\begin{equation}\label{eq:opt_poly}
    p(\mathbf{x}) = \max_{\mathbf{w} \in W}(p_{\mathbf{w}} + \mathbf{w}^\top \mathbf{x}),
\end{equation}
 where the vector of coefficients is

\begin{equation}\label{eq:opt_coeffs}
\mathbf{p} = (p_{\mathbf{w}})_{\mathbf{w}\in W} = (-\mathbf{X}^\top)\boxplus'\mathbf{y} + \frac{1}{2}\left\|\mathbf{X}\boxplus((-\mathbf{X}^\top)\boxplus' \mathbf{y}) - \mathbf{y}\right\|_{\infty}.
\end{equation}

\paragraph{Variants of Tropical Polynomial Regression}

In addition to the closed form solution of the tropical regression problem in the $\infty$-norm described above, other variants of tropical polynomial regression have been explored recently. In \cite{HOOK2019maxplus2norm}, the authors present two algorithms to solve the 2-norm max plus regression problem. The first involves a brute force search over sparsity patterns of solution vectors and the second involves a variant of Newton's method. In \cite{magnani_convex_2009}, the authors present a method for convex piecewise linear regression with the 2-norm loss which involves iteratively partitioning the input data and fitting affine functions to each partition. Finally, in \cite{tsiamis2019sparsity,tsilivis2022towardSparsityWeighted}, the authors leverage the weighted lattice framework above to find sparse solutions to the problem \eqref{eq:inf_norm_sol}. Specifically, the authors present a greedy algorithm to find a sparse solution to \eqref{eq:subsol_approx} for $p < \infty$ then shift the finite entries by half of the infinity norm of the residual. 

\subsection{Relations to ReLU Neural Networks}

This section fixes notation for and  briefly overviews the relationship between neural networks and tropical algebraic objects. We closely follow the presentation in \cite{zhang2018tropical}. A recent survey on tropical algebraic techniques for machine learning is  \cite{Maragos2021TropML}.

An \emph{$L$-layer neural network} with \emph{ReLU} activation functions is a function $\nu:\bbR^n \to \bbR$ that can be expressed as a composition of functions

\[\nu =  \rho^{(L)} \circ \sigma^{(L-1)} \circ \rho^{(L-1)} \circ \sigma^{(L-2)}\circ \cdots \circ \sigma^{(1)}\circ \rho^{(1)},\]
 where $\sigma^{(\ell)}(\mathbf{x}) = \max(\mathbf{x},0)$ is the ReLU activation function and $\rho^{(\ell)}(\mathbf{x}) = \mathbf{A}^{(\ell)}\mathbf{x}  + \mathbf{b}^{(\ell)}$ is affine. The matrix $\mathbf{A}^{(\ell)}$ and the vector $\mathbf{b}^{(\ell)}$ encode the \emph{weights} and \emph{bias} of layer $\ell$, respectively. ReLU neural networks are continuous and piecewise linear by construction. Under assumptions on the entries of the $\mathbf{A}^{(\ell)}$, they can additionally be written as tropical rational functions. 

\begin{theorem}[\cite{zhang2018tropical}]\label{thm:fun_equiv}
The following classes of functions are the same

\begin{enumerate}
    \item[(i)] Tropical rational functions
    \item[(ii)] Continuous piecewise linear functions with integer coefficients
    \item[(iii)] ReLU neural networks with integer weights
\end{enumerate}

\end{theorem}

The authors of \cite{zhang2018tropical} note that if $\nu$ is a neural network with nonintegral weights, then rounding weights to rational numbers and clearing denominators gives a network with integer weights. More precisely,

\begin{corollary}\label{cor:scal_rat}
If $\nu: \bbR^n \to \bbR$ is a ReLU neural network, then there is a real number $c$ and a tropical rational function $f \in \mathbb{T}(x_1,\ldots, x_n)$ such that $\nu(\mathbf{x})$ is approximated arbitrarily closely by $f(c\mathbf{x})$ for all $\mathbf{x} \in \bbR^n$. 
\end{corollary}

Removing the integrality condition on the weights of $\nu$ gives the following result:

\begin{theorem}[\cite{arora2018understanding}]
If $\nu:\bbR^n \to \bbR$ is a ReLU neural network, then $\nu$ is a piecewise linear function. Conversely, if $r:\bbR^n \to \bbR$ is piecewise a piecewise linear function, then $r$ can be represented as a ReLU neural network with at most $\lceil \log_2(n+1)\rceil + 1$ layers. 
\end{theorem}

\section{Alternating Method For Tropical Rational Regression}\label{sec:Algorithm}

We adapt the polynomial regression method described in Section \ref{sec:Lattice_Regression} to fit tropical rational functions to a dataset. Recall that a tropical rational function is a function of the form $f(\mathbf{x}) := p(\mathbf{x})-q(\mathbf{x})$, where $p$ and $q$ are tropical polynomials. So, for some finite subset $W \subseteq \bbZ^n_{\geq 0}$,

\[f(\mathbf{x}) = p(\mathbf{x}) - q(\mathbf{x}) = \max_{\mathbf{w} \in W}({\bf w}^\top {\bf x} + p_\mathbf{w}) - \max_{\mathbf{w} \in W}({\bf w}^\top \mathbf{x} + q_\mathbf{w}).\]

Given a set of points $\mathbf{x}^{(1)},\mathbf{x}^{(2)},\ldots, \mathbf{x}^{(N)} \in \bbR^n$, set $\mathbf{X}$ to be the matrix whose rows are indexed by $i \in \{1,\ldots, N\}$ and columns are indexed by $\mathbf{w} \in W$ with $\mathbf{X}_{(i,\mathbf{w})} = \mathbf{w}^\top \mathbf{x}^{(i)}$.  Evaluation of the tropical rational function $f$ at the points $\mathbf{x}^{(i)}$ is then given by 

\begin{equation}\label{eq:rational_as_mvps}
\begin{bmatrix} f(\mathbf{x}^{(1)}) & f(\mathbf{x}^{(2)}) & \cdots & f(\mathbf{x}^{(N)})\end{bmatrix}^\top = \mathbf{X} \boxplus \mathbf{p} - \mathbf{X} \boxplus \mathbf{q},
\end{equation}
 where $\mathbf{p} = (p_\mathbf{w})_{\mathbf{w} \in W}$ and $\mathbf{q} = (q_\mathbf{w})_{\mathbf{w} \in W}$ are the vectors of coefficients of $p$ and $q$. Using the representation \eqref{eq:rational_as_mvps}, it follows that we can rewrite the problem \eqref{eq:ProblemDef} as

\begin{equation}\label{eq:Opt_problem}
  \arg\min_{f \in \mathbb{T}(\mathbf{x})_W }\left\|\begin{bmatrix}f(\mathbf{x}^{(1)}) & f(\mathbf{x}^{(2)}) & \cdots & f(\mathbf{x}^{(N)})\end{bmatrix}^\top - \mathbf{y}\right\|_{\infty} = \arg\min_{\mathbf{p},\mathbf{q}}\left\| \mathbf{X} \boxplus \mathbf{p} - \mathbf{X} \boxplus \mathbf{q} - \mathbf{y}\right\|_{\infty}.  
\end{equation}

For fixed $\mathbf{q}$, the problem $\arg \min_{\mathbf{p}}\|\mathbf{X} \boxplus \mathbf{p} - (\mathbf{X} \boxplus \mathbf{q} + \mathbf{y})\|_\infty$ is a tropical polynomial regression problem. By Theorem \ref{thm:opt_sol}, this problem has the analytical solution 

\[\mathbf{p}_{*}(\mathbf{q}) = (-\mathbf{X}^\top) \boxplus' (\mathbf{X} \boxplus \mathbf{q} + \mathbf{y}) + \frac{1}{2}\left\|\mathbf{X} \boxplus \left((-\mathbf{X}^\top) \boxplus' (\mathbf{X} \boxplus \mathbf{q} + \mathbf{y})\right) - (\mathbf{X} \boxplus \mathbf{q} + \mathbf{y}) \right\|_{\infty}.\]
 Similarly, for fixed $\mathbf{p}$, the problem $\arg\min_{\mathbf{q}}\|\mathbf{X}\boxplus \mathbf{q} - (\mathbf{X} \boxplus \mathbf{p} - \mathbf{y})\|$ has the analytical solution

\[\mathbf{q}_{*}(\mathbf{p}) = (-\mathbf{X}^\top) \boxplus' (\mathbf{X} \boxplus \mathbf{p} - \mathbf{y}) + \frac{1}{2}\left\|\mathbf{X} \boxplus \left((-\mathbf{X}^\top) \boxplus' (\mathbf{X} \boxplus \mathbf{p} - \mathbf{y})\right) - (\mathbf{X} \boxplus \mathbf{p} - \mathbf{y}) \right\|_{\infty}.\]

Moreover, these analytical solutions can be found quickly, as they rely only on max-plus and min-plus matrix-vector products and do not need to solve a linear system. We exploit this to search over the space of tropical rational functions by alternating between fitting the numerator polynomial and the denominator polynomial. This method is summarized below as Algorithm \ref{alg:alt_fit}. 

\begin{algorithm}[t]
\caption{Alternating fit for tropical rational functions \label{alg:alt_fit}}
\KwIn{Dataset $\cD = (\mathbf{x}^{(i)},y^{(i)})_{i = 1}^{N} \subseteq \bbR^{n} \times  \bbR$,\\ Set of permissible exponents $W \subseteq \bbZ_{\geq 0}^n$, \\ Maximum number of iterations $k_{\max}$}
\KwOut{Vectors $\mathbf{p}$ and $\mathbf{q}$ of coefficients of tropical polynomials $p,q \in \mathbb{T}[\mathbf{x}]_W$ such that $p(\mathbf{x}^{(i)}) - q(\mathbf{x}^{(i)}) \approx y^{(i)}$ }
\smallskip

\smallskip
Set $\mathbf{X} \in \bbR^{N \times |W|}$ to be the matrix with entries $\mathbf{X}_{i,\mathbf{w}} = (\mathbf{w}^\top \mathbf{x}^{(i)})$\;
$\mathbf{p}^0,\mathbf{q}^0\gets -\infty, \mathbf{q}^0_{\mathbf{0}} \gets -\text{mean}(\mathbf{y})$\;
\For{$k \leq k_{max}$}{
    $\mathbf{p}^k \gets \arg \min_{\mathbf{p}}\|\mathbf{X} \boxplus \mathbf{p} - \mathbf{X}\boxplus \mathbf{q}^{k-1} - \mathbf{y}\|_{\infty}$\;
    $\mathbf{q}^k \gets \arg \min_\mathbf{q}\|\mathbf{X}\boxplus \mathbf{p}^k - \mathbf{X}\boxplus \mathbf{q} - \mathbf{y}\|_{\infty}$\;
}
$\mathbf{p} \gets \mathbf{p}^{k_{\max}}$; $\mathbf{q} \gets \mathbf{q}^{k_{\max}}$
\end{algorithm}

While Algorithm \ref{alg:alt_fit} is defined for general choices of $W \subseteq \bbZ^n_{\geq 0}$, our implementation takes $W$ to be of the form $W = \{(w_1,w_2,\ldots, w_n) \in \bbZ^n_{\geq 0} : w_i \leq d_i\}$ for some $\mathbf{d} = (d_1, d_2,\ldots, d_n)^\top \in \bbZ^n_{\geq 0}$. In this case, $|W|$ becomes very large if $\mathbf{d}$ has large entries or if $n$ is large. In \cite{maragos2019tropical}, the authors discuss choosing $W$ by clustering approximated gradients from the data to reduce the number of parameters used in fitting tropical polynomials. The choice of initialization is such that $f$ is initialized to the constant function $f(x_1,\ldots, x_n) = \text{mean}(\mathbf{y})$.

As a step towards understanding convergence properties of Algorithm \ref{alg:alt_fit}, we show that the error at each iteration is nonincreasing. 

\begin{proposition}
\label{prop:error_nonincreasing}
The error $e^k = \|\mathbf{X}\boxplus \mathbf{p}^k - \mathbf{X} \boxplus \mathbf{q}^{k} - \mathbf{y}\|_{\infty}$ is nonincreasing. 
\end{proposition}

\begin{proof}
We show that $e^{k+1}\leq e^k$ for any $k$. By construction, $\mathbf{p}^{k+1}$ satisfies 

\[\|\mathbf{X}\boxplus \mathbf{p}^{k+1} - \mathbf{X}\boxplus \mathbf{q}^{k} - \mathbf{y}\|_{\infty} \leq \|\mathbf{X}\boxplus \mathbf{p}^{k} - \mathbf{X}\boxplus \mathbf{q}^{k} - \mathbf{y}\|_{\infty} = e^k.\]
 Similarly, 

\[e^{k+1} = \|\mathbf{X}\boxplus \mathbf{p}^{k+1} - \mathbf{X}\boxplus \mathbf{q}^{k+1} - \mathbf{y}\|_{\infty} \leq \|\mathbf{X}\boxplus \mathbf{p}^{k+1} - \mathbf{X}\boxplus \mathbf{q}^{k} - \mathbf{y}\|_{\infty}.\]
 It then follows that $e^{k+1} \leq e^k$.\end{proof}

The decrease in error between iterations is bounded by a constant multiple of the norm of the update step. 

\begin{proposition}\label{prop:error_bounded_by_updates}
Let $\eta^k = \left\|\begin{bmatrix}\mathbf{p}^{k+1} & \mathbf{q}^{k+1}\end{bmatrix}^\top - \begin{bmatrix}\mathbf{p}^{k} & \mathbf{q}^{k}\end{bmatrix}^\top\right\|_\infty$. Then, the change in error between iterations $e^k - e^{k+1}$ is bounded:

\[e^k - e^{k+1} \leq 2\eta^k.\]
\end{proposition}

\begin{proof}

First, note that for each $\mathbf{w} \in W$, 

\[p_\mathbf{w}^k - \eta^k \leq p_\mathbf{w}^{k+1} \leq p_\mathbf{w}^{k} + \eta^k \quad \text{ and } \quad  q_\mathbf{w}^k - \eta^k \leq q_\mathbf{w}^{k+1} \leq q_\mathbf{w}^{k} + \eta^k.\]

Because for fixed $i \in \{1,2,\ldots, N\}$,  $\max_{\mathbf{w} \in W}(p_\mathbf{w}^k + \mathbf{w}^\top \mathbf{x}^{(i)}) -\eta^k = \max_{\mathbf{w} \in W}(p_\mathbf{w}^k -\eta^k + \mathbf{w}^\top \mathbf{x}^{(i)})$, this in turn implies that

\[\max_{\mathbf{w} \in W}(p_\mathbf{w}^k + \mathbf{w}^\top \mathbf{x}^{(i)})  -\eta^k\leq \max_{\mathbf{w} \in W}(p_\mathbf{w}^{k+1} + \mathbf{w}^\top \mathbf{x}^{(i)}) \leq \max_{\mathbf{w} \in W}(p_\mathbf{w}^k + \mathbf{w}^\top \mathbf{x}^{(i)}) + \eta^k. \]
 The analogous statement holds with $\mathbf{q}^k$ replacing $\mathbf{p}^k$. 

Let $\ell \in \{1,2,\ldots, N\}$ be such that $e^{k} = |\max_{\mathbf{w} \in W}(p_\mathbf{w}^k + \mathbf{w}^\top \mathbf{x}^{(\ell)}) - \max_{\mathbf{w} \in W}(q_\mathbf{w}^k + \mathbf{w}^\top \mathbf{x}^{(\ell)}) - y^{(\ell)}|$. Now,

\[\begin{aligned}
e^k - e^{k+1} &= |p^k(\mathbf{x}^{(\ell)}) - q^k(\mathbf{x}^{(\ell)}) - y^{(\ell)}| - \max_{i = 1,2,\ldots, N}|p^{k+1}(\mathbf{x}^{(i)}) - q^{k+1}(\mathbf{x}^{(i)}) - y^{(i)}|\\
&\leq |p^k(\mathbf{x}^{(\ell)}) - q^k(\mathbf{x}^{(\ell)}) - y^{(\ell)}| - |p^{k+1}(\mathbf{x}^{(\ell)}) - q^{k+1}(\mathbf{x}^{(\ell)}) - y^{(\ell)}|\\
&\leq |p^k(\mathbf{x}^{(\ell)}) - q^k(\mathbf{x}^{(\ell)}) - p^{k+1}(\mathbf{x}^{(\ell)}) + q^{k+1}(\mathbf{x}^{(\ell)})|\\
&\leq |p^k(\mathbf{x}^{(\ell)}) - p^{k+1}(\mathbf{x}^{(\ell)})| + |q^k(\mathbf{x}^{(\ell)}) - q^{k+1}(\mathbf{x}^{(\ell)})|\\
&\leq 2\eta^k
\end{aligned}\]
\end{proof}

Empirically, the term $\eta^k$ appears to be nonincreasing (see Section \ref{sec:Expirements}). This motivates the use of a sufficiently low value of $\eta^k$ as a stopping criterion in Algorithm \ref{alg:alt_fit}. 

\subsection{Nondifferentiability of the Loss Function}

In this section, we investigate the geometry of the problem \eqref{eq:ProblemDef} by viewing the loss function as a tropical rational function. In particular, we show that \eqref{eq:ProblemDef} always has a minimizer for which the loss function is nondifferentiable. Moreover, the iterates produced by Algorithm \ref{alg:alt_fit} are always elements of the nondifferentiability locus of the loss function. Finally, we discuss preliminary consequences of nondifferentiability at a minimizer.

\begin{proposition}\label{prop:loss_rat}
The loss function 

\[\mathcal{L}(\mathbf{p},\mathbf{q}) = \left\|\mathbf{X} \boxplus \mathbf{p} - \mathbf{X}\boxplus \mathbf{q} - \mathbf{y}\right\|_{\infty}\]
is a tropical rational function of the coefficients $p_\mathbf{w},q_{\mathbf{w}}$.
\end{proposition}

\begin{proof}
Note that 

\[\begin{aligned}
\mathcal{L}(\mathbf{p},\mathbf{q}) &= \left\|\mathbf{X} \boxplus \mathbf{p} - \mathbf{X}\boxplus \mathbf{q} - \mathbf{y}\right\|_{\infty}\\
&= \max_{i = 1,2,\ldots, N} \left[\max(p(\mathbf{x}^{(i)}) - q(\mathbf{x}^{(i)})- y^{(i)}, y^{(i)} + q(\mathbf{x}^{(i)}) - p(\mathbf{x}^{(i)}))\right].
\end{aligned}\]

Now, for each $i = 1,2,\ldots, N$, the evaluation map on $\mathbb{T}[\mathbf{x}]_W$ which sends $g \mapsto g(\mathbf{x}^{(i)})$ is tropically linear in the coefficients $g_\mathbf{w}$. In particular, for each $i = 1,2,\ldots, N$, both $p(\mathbf{x}^{(i)}) - q(\mathbf{x}^{(i)})- y^{(i)}$ and $y^{(i)} + q(\mathbf{x}^{(i)}) - p(\mathbf{x}^{(i)})$ are tropical rational functions of the parameters $p_{\mathbf{w}},q_{\mathbf{w}}$. Because the set of tropical rational functions is closed under tropical addition, this implies that $\mathcal{L}(\mathbf{p},\mathbf{q})$ is a tropical rational function.
\end{proof}

Proposition \ref{prop:loss_rat} allows us to use the polyhedral geometry of tropical hypersurfaces to study the geometry of the optimization problem \eqref{eq:ProblemDef}. 

\begin{proposition}\label{prop:nondiff_exist}
There is an optimal solution to \eqref{eq:ProblemDef}. Moreover, there is an optimal solution $(\mathbf{p}^*,\mathbf{q}^*)$ such that $\nabla \mathcal{L}(\mathbf{p}^*,\mathbf{q}^*)$ does not exist. 
\end{proposition}

\begin{proof}
By Proposition \ref{prop:loss_rat}, there are tropical polynomials $g,h$ in $2|W|$ indeterminates such that

\[\mathcal{L}(\mathbf{p},\mathbf{q}) = g(\mathbf{p},\mathbf{q}) - h(\mathbf{p},\mathbf{q}).\]
The nondifferentiability locus of $\mathcal{L}$ is a subset of $\Sigma = \cV(g)\cup \cV(h) \subseteq \bbR^{2|W|}$. There are finitely many connected components of $\bbR^{2|W|}\setminus \Sigma$, each of which are open polyhedra. Label these polyhedra $A_1,A_2,\ldots, A_s$. 

For the first claim, note that $\cL$ is linear on $\mathrm{cl}(A_i)$ for each $i = 1,\ldots, s$. Because $\cL(\mathbf{p},\mathbf{q})\geq 0$ for all $(\mathbf{p},\mathbf{q}) \in \bbR^{2|W|}$, the restriction of $\cL$ to $\mathrm{cl}(A_i)$ achieves a minimum value $z_i$ on $\mathrm{cl}(A_i)$. Then $\cL$ achieves the minimum value $z = \min_{i = 1,2,\ldots, s} z_i$. 

For the second claim, note that the restriction of $\cL$ to $\mathrm{cl}(A_i)$ achieves its minimum on the boundary $\partial A_i$ for each $i = 1,2\ldots, s$. So, there must be an optimal solution in $\Sigma = \cup_{i = 1}^s \partial A_i$. Let $(\hat{\mathbf{p}},\hat{\mathbf{q}})$ be an optimal solution in $\Sigma$ such that $\nabla \mathcal{L}(\hat{\mathbf{p}},\hat{\mathbf{q}})$ exists. By the hypothesis that $(\hat{\mathbf{p}},\hat{\mathbf{q}})$ is an optimal solution, it is necessary that $\nabla \mathcal{L}(\hat{\mathbf{p}},\hat{\mathbf{q}}) = 0$. 
Relabeling the $A_i$ if necessary, let $A_1,\ldots, A_k$ be such that $(\hat{\mathbf{p}},\hat{\mathbf{q}}) \in \cap_{i = 1}^{k}\mathrm{cl}(A_i)$.  Because $\mathcal{L}$ is linear on each $A_i$, it must be the case that $\nabla \mathcal{L}\vert_{A_i} = 0$ for each $i$. This implies that every point in $A = \cup_{i = 1}^k \mathrm{cl}(A_i)$ is a minimizer of $\mathcal{L}$. Set $B$ to be the smallest connected subset containing $A$ on which $\mathcal{L}$ is minimized. Note that $B = \cup_{i = 1}^r \mathrm{cl}(A_i)$ where $r \geq k$. If $B \not = \bbR^{2|W|}$, then there is a point $(\mathbf{p}^*,\mathbf{q}^*)$ on the boundary of $B$ where $\nabla \mathcal{L}$ does not exist. Otherwise, $B = \bbR^{2|W|}$ and therefore $\mathcal{L}$ is constant. However, $\cL$ cannot be constant because for fixed $\mathbf{w}\in W$, fixed $\mathbf{q}$, and fixed $p_{\mathbf{v}}$ for $\mathbf{v} \not = \mathbf{w}$, 

\[\cL(\mathbf{p},\mathbf{q}) = \max_{i = 1,2,\ldots, N}\left|\mathbf{w}^\top\mathbf{x}^{(i)} + p_{\mathbf{w}} - \max_{\mathbf{w}}(\mathbf{w}^\top \mathbf{x}^{(i)} + q_{\mathbf{w}}) - y^{(i)}\right| = p_{\mathbf{w}} - C\] 
for some constant $C$ for sufficiently large values of $p_{\mathbf{w}}$. \end{proof}

There are problems for which every optimal solution is in the nondifferentiability locus of $\mathcal{L}$.

\begin{example}\label{ex:all_nondiff}
Consider the problem with $\mathcal{D} = \{(0,0)\}$, $W = \{0\}\subseteq \mathbb{Z}$. A tropical rational function with $W$ as its support is 

\[f(x) = p_0 - q_0.\]
 The minimum of $\mathcal{L}$ is $0$, achieved on the line $\spann\{(1,1)\}$. However, along this line, the loss is 

\[\mathcal{L}(p_0,q_0) = \|p_0 - q_0\|_\infty = |p_0-q_0|\]
so that $\nabla \mathcal{L}(p_0,q_0)$ does not exist when $p_0 = q_0$. 
\end{example}

Algorithm \ref{alg:alt_fit} produces iterates in the nondifferentiablity locus of $\mathcal{L}$.

\begin{theorem}
The gradient $\nabla \mathcal{L}(\mathbf{p}^k,\mathbf{q}^k)$ does not exist, where $\mathbf{p}^k$ and $\mathbf{q}^k$ are defined as in Algorithm \ref{alg:alt_fit} and $k\geq 1$. 
\end{theorem}

\begin{proof}
For fixed $k \geq 1$, the functions $\mathcal{L}(\mathbf{p},\mathbf{q}^{k-1})$ and $\mathcal{L}(\mathbf{p}^k,\mathbf{q})$ are the infinity norm of the residual of a tropical polynomial regression problem. Because the updates $\mathbf{p}^k$ and $\mathbf{q}^k$ are minimizers of the infinity norm of such residuals, it suffices to show that in the setup of Theorem \ref{thm:opt_sol}, 

\[\mathbf{u}^* = \arg \min_{\mathbf{u}} \|\mathbf{A}\boxplus \mathbf{u} - \mathbf{b}\|_{\infty}\]
 is a nondifferentiable point of the function $\mathcal{R}(\mathbf{u}) = \|\mathbf{A}\boxplus \mathbf{u} - \mathbf{b}\|_{\infty}$.

Suppose for the sake of a contradiction that $\nabla \mathcal{R}(\mathbf{u}^*)$ exists. Because $\mathbf{u}^*$ minimizes $\cR$ by hypothesis, it follows that $\nabla \cR (\mathbf{u}^*) = 0$. Fix indices $i$ and $j$ such that 

\[\cR(\mathbf{u}^*) = |\max_{\ell}(a_{i,\ell} + u_\ell^*) - b_i| = |a_{i,j} + u_j^* - b_i|.\]
 Set $J = \{k\;\vert\; a_{i,k} + u_k^* = \max_{\ell}(a_{i,\ell} + u_\ell^*)\}$ and $\mathbf{e}_J$ to be the vector with 1 in component $k$ if $k \in J$ and $0$ otherwise. Note that the fixed index $j\in J$. Then, if $\epsilon > 0$ is small enough that $a_{i,k} + u_k^* - \epsilon > a_{i,\ell} + u_\ell^*$ when $k \in J$ and $\ell \not \in J$, then there exists $c \in \{-1,1\}$ such that 

\[\cR(\mathbf{u}^* + c\epsilon \mathbf{e}_J) \geq |a_{i,j} +u_j^* + c\epsilon - b_i| = |a_{i,j} + u_j^* - b_i| + \epsilon.\]
 But then, the difference quotient

\[\left|\frac{\cR(\mathbf{u}^* + c\epsilon \mathbf{e}_J) - \cR(\mathbf{u}^*)}{\epsilon}\right| \geq 1\]
 is bounded away from 0 for $\epsilon > 0$ sufficiently small, a contradiction with the hypothesis that $\nabla \cR(\mathbf{u}^*) = 0$.

\end{proof}

Proposition \ref{prop:nondiff_exist} and Example \ref{ex:all_nondiff} demonstrate the importance of understanding the nondifferentiability locus of $\mathcal{L}$. The two sources of nondifferentiability in $\mathcal{L}$ are the nondifferentiability of $\|\mathbf{u} - \mathbf{v}\|_{\infty}$ as a function of $\mathbf{u}$ and the nondifferentiability of the tropical rational functions $f(\mathbf{x}^{(i)}) = p(\mathbf{x}^{(i)})-q(\mathbf{x}^{(i)})$ as a function of the $p_\mathbf{w}$ and $q_\mathbf{w}$. This connects the geometry of the dataset to that of a tropical rational function produced as an iterate of Algorithm \ref{alg:alt_fit} by providing a certificate that $(\mathbf{p},\mathbf{q})$ is in the nondifferentiability locus of $\cL$ in terms of the input data.

\begin{proposition}\label{prop:Geometry_opt}
There exists a minimizer $(\mathbf{p}^*,\mathbf{q}^*)$ of $\mathcal{L}$ such that at least one of the following holds:

\begin{enumerate}
\item There is an $i$ such that $\mathbf{x}^{(i)} \in \cV(p^*)\cup \cV(q^*)$  
\item The infinity norm in $\mathcal{L}(\mathbf{p}^*,\mathbf{q}^*)$ is achieved by at least two data points $(\mathbf{x}^{(i)},y^{(i)}),(\mathbf{x}^{(j)},y^{(j)})$.
\end{enumerate}
\end{proposition}

\begin{proof}
We show the contrapositive. Let $(\mathbf{p}^*,\mathbf{q}^*)$ be a minimizer of $\mathcal{L}$ such that $\nabla \mathcal{L}(\mathbf{p}^*,\mathbf{q}^*)$ does not exist and suppose that neither condition holds. Let $1\leq i\leq N$ be such that 

\[\cL(\mathbf{p}^*,\mathbf{q}^*)=\left|\max_{\mathbf{w}\in W}(p^*_\mathbf{w} + \mathbf{w}^\top \mathbf{x}^{(i)}) - \max_{\mathbf{w}\in W}(q^*_\mathbf{w} + \mathbf{w}^\top \mathbf{x}^{(i)})-y^{(i)}\right| > \left|\max_{\mathbf{w}\in W}(p^*_\mathbf{w} + \mathbf{w}^\top \mathbf{x}^{(j)}) - \max_{\mathbf{w}\in W}(q^*_\mathbf{w} + \mathbf{w}^\top \mathbf{x}^{(j)}) - y^{(j)}\right|\]
 for all $j \not = i$. Then there is an open neighborhood $U \subseteq \bbR^{2|W|}$ of $(\mathbf{p}^*,\mathbf{q}^*)$ such that

\[\mathcal{L}(\mathbf{p},\mathbf{q}) = \left|\max_{\mathbf{w}\in W}(p_\mathbf{w} + \mathbf{w}^\top \mathbf{x}^{(i)}) - \max_{\mathbf{w}\in W}(q_\mathbf{w} + \mathbf{w}^\top \mathbf{x}^{(i)})-y^{(i)}\right|\]
 for all $(\mathbf{p},\mathbf{q})\in U$. Because $\cL(\mathbf{p}^*,\mathbf{q}^*) > 0$, it suffices to show that if $\mathbf{x}^{(i)} \not \in \cV(p^*) \cup \cV(q^*)$ then the evaluation map $(\mathbf{p},\mathbf{q})\mapsto p(\mathbf{x}^{(i)}) - q(\mathbf{x}^{(i)})$ is differentiable at $(\mathbf{p}^*,\mathbf{q}^*)$. By the hypothesis that $\mathbf{x}^{(i)} \not \in \cV(p^*) \cup \cV(q^*)$, there are $\mathbf{w}_1,\mathbf{w}_2$ such that $p^*_{\mathbf{w}_1} + \mathbf{w}_1^\top\mathbf{x}^{(i)} > p^*_{\mathbf{w}} + \mathbf{w}^\top\mathbf{x}^{(i)}$ for all $\mathbf{w}_1 \not = \mathbf{w} \in W$ and $q^*_{\mathbf{w}_2} + \mathbf{w}_2^\top\mathbf{x}^{(i)} > q^*_{\mathbf{w}} + \mathbf{w}^\top\mathbf{x}^{(i)}$ for all $\mathbf{w}_2 \not = \mathbf{w} \in W$. Restricting $U$ to a smaller open neighborhood of $(\mathbf{p}^*,\mathbf{q}^*)$ if necessary then gives that 

\[\cL(\mathbf{p},\mathbf{q}) = \left|p_{\mathbf{w}_1} + \mathbf{w}_1^\top \mathbf{x}^{(i)} - q_{\mathbf{w}_1} + \mathbf{w}_2^\top \mathbf{x}^{(i)} -y^{(i)} \right|\]
 near $(\mathbf{p}^*,\mathbf{q}^*)$. This is an affine function of $(\mathbf{p},\mathbf{q})$ on $U$ because $\cL(\mathbf{p},\mathbf{q}) > 0$ and therefore $\nabla \cL (\mathbf{p}^*,\mathbf{q}^*)$ exists.

\end{proof}

A further exploration of geometric conditions relating optimal functions to the training data would be interesting, but we do not pursue this line of inquiry further in the current work.  

\subsection{Polynomial Evaluation}\label{sec:Poly_eval}  In order to effectively use Algorithm \ref{alg:alt_fit}, we need to be able to efficiently perform the matrix-vector multiplications involved in solving the minimization problem. This amounts to evaluating a  tropical polynomial and performing a min-plus matrix-vector product using the negative transpose of a ``Vandermonde" type matrix. 

\paragraph{Univariate Polynomial Evaluation}
Let $n=1$ and consider the case of evaluating the degree $d$ univariate tropical polynomial $p(t) = \max_{0\leq j \leq d}(jt + p_j)$ at the points $x^{(1)},x^{(2)},\ldots, x^{(N)} \in \bbR$. In this case the matrix $\mathbf{X}$ is given by $\mathbf{X} = \begin{bmatrix} \mathbf{0} & \mathbf{x} & \cdots &d\mathbf{x}\end{bmatrix}$, where $\mathbf{x}$ is the vector of the $x^{(i)}$. Then, $\mathbf{v} = \mathbf{X}\boxplus \mathbf{p}$ a vector of evaluations of $p$ at the $x^{(i)}$. This computation does not require the explicit formation of the highly structured matrix $\mathbf{X}$. The solution $\mathbf{v}$ can be computed by setting $\mathbf{v}^{0} = p_0\mathbf{1}$ and computing 

\[\mathbf{v}^k = \max(\mathbf{v}^{k-1},k\mathbf{x} + p_k\mathbf{1}), \quad 1\leq k\leq d,\]
 so that $\mathbf{v}^d = \mathbf{v}$. This approach avoids the construction of the $N \times (d+1)$ matrix $\mathbf{X}$ and instead only uses the length $N$ vector $\mathbf{x}$.

Similarly, an explicit construction of the matrix $\mathbf{X}$ can be avoided when computing $\hat{\mathbf{p}} = (-\mathbf{X}^\top) \boxplus' \mathbf{y}$. This follows because $\hat{p_j} = \min_{1\leq i \leq N}(y_i - jx_i)$, so that there is no need to construct $\mathbf{X}$.

Finally, to compute $\mathbf{v} = \mathbf{X}\boxplus \left((-\mathbf{X})^\top \boxplus' \mathbf{y}\right)$, we set $\mathbf{v}^0 = \min_{1\leq i\leq N}(y_i)\mathbf{1}$ and compute 

\[\mathbf{v}^k = \max\left(\mathbf{v}^{k-1},k\mathbf{x} + \min_{1\leq i \leq N}(y_i - kx_i)\mathbf{1}\right), \quad 1\leq k\leq d.\]
Then, $\mathbf{v}^d = \mathbf{v} = \mathbf{X}\boxplus \left((-\mathbf{X})^\top \boxplus' \mathbf{y}\right)$. 

The methods in the univariate case form the basis for effective computations with multivariate tropical polynomials as the number of columns in the matrix $\mathbf{X}$ grows as $d^n$. 

\paragraph{Multivariate Polynomial Evaluation}

We extend the univariate polynomial evaluation method to the multivariate case by considering a polynomial $p\in \mathbb{T}[x_1,\ldots,x_n]$ as a polynomial in the variable $x_n$ with coefficients in $\mathbb{T}[x_1,\ldots, x_{n-1}]$ and evaluating the coefficients.

For example, in the bivariate case, the polynomial

\[p(x_1,x_2) = \max_{0\leq i\leq d_1, 0\leq j \leq d_2}(ix_1 +jx_2 + p_{i,j})\]
 is to be evaluated at a given set of evaluation points $\left(x_1^{(1)},x_2^{(1)}\right),\left(x_1^{(2)},x_2^{(2)}\right),\ldots, \left(x_1^{(N)},x_2^{(N)}\right) \in \mathbb{R}^2$. Rewrite the polynomial $p$, collecting all terms of the same degree in $x_2$. Using tropical notation, this gives

\[p(x_1,x_2) = \bigoplus_{j = 0}^{d_2} x_2^{\odot j}\odot\left(\bigoplus_{i = 0}^{d_1} x_1^{\odot i} \odot p_{i,j}\right) = \max_{j = 0,\ldots, d_2} \left(jx_2 + \max_{i = 0,\ldots, d_1} (ix_1 + p_{i,j})\right).\]

Now, the term $\max_{i = 0,\ldots, d_1} (ix_1 + p_{i,j})$ is a univariate tropical polynomial for each $j$ and can therefore be evaluated without the construction of the matrix $\mathbf{X}$. Once these terms are each evaluated, $p$ is a univariate polynomial in $x_2$. 
Ultimately, this avoids the construction of the large $N \times (d_1 + 1)(d_2 + 1)$ matrix $\mathbf{X}$ and instead only uses the $N$ pairs $\left(x_1^{(i)},x_2^{(i)}\right)$ and the degree bounds $d_1,d_2$.

We also compute the solution to the polynomial subfit problem without explicitly constructing the matrix $\mathbf{X}$. Similarly to the univariate case, each entry in the output of $\hat{\mathbf{p}} =(-\mathbf{X}^\top) \boxplus' \mathbf{y}$ has the form $$\hat{p}_{\mathbf{w}}=\min_{1\leq i\leq N}\left(y_i - \sum_{j = 1}^{n}w_jx_j^{(i)}\right).$$ So, it is not necessary to store more than the evaluation points $\left(x_1^{(i)},x_2^{(i)}\right)$. 

Finally, to evaluate the product $\mathbf{v} = \mathbf{X}\boxplus\left((-\mathbf{X})^\top \boxplus' \mathbf{y}\right)$, we initialize $\mathbf{v}^{0} = \min_{1\leq i \leq N}(y_i) \mathbf{1}$ and set $D = \prod_{\ell = 1}^{n}(d_\ell+1)-1$. For an enumeration $W\setminus \{0\} = \left\{\mathbf{w}^{(1)},\mathbf{w}^{(2)},\ldots, \mathbf{w}^{(D)}\right\}$ set 

\[\mathbf{u}^k = \begin{bmatrix}\sum_{j = 1}^{n}w_j^{(k)}x_j^{(1)} & \sum_{j = 1}^{n}w_j^{(k)}x_j^{(2)} & \cdots & \sum_{j = 1}^{n}w_j^{(k)}x_j^{(N)} \end{bmatrix}^\top, \quad k = 1,2,\ldots, D\]
and update 
\[\mathbf{v}^k = \max\left(\mathbf{v}^{k-1}, \mathbf{u}^k + \min_{1\leq i \leq N}(y_i - u^k_i)\mathbf{1}\right), \quad k = 1,2,\ldots, D.\]
Then $\mathbf{v}^{D} = \mathbf{v} = \mathbf{X}\boxplus\left((-\mathbf{X})^\top \boxplus' \mathbf{y}\right)$. Note that if $\mathbf{w}^{(k)} - \mathbf{w}^{(k-1)}$ is the standard basis vector $\mathbf{e}_j$, then $\mathbf{u}^k$ can be constructed from $\mathbf{u}^{k-1}$ as $\mathbf{u}^k = \mathbf{u}^{k-1} + \begin{bmatrix}x_j^{(1)} & x_j^{(2)} & \cdots & x_j^{(N)}\end{bmatrix}^\top $ and therefore the updates to $\mathbf{u}^{k}$ and $\mathbf{v}^{k}$ can both be computed efficiently from the input data. 

\section{Computational Experiments}\label{sec:Expirements}

In this section, we use Algorithm \ref{alg:alt_fit} for regression tasks and examine its convergence behavior empirically. We provide univariate, bivariate, and higher dimensional examples. In the univariate case we analyze the relationship between the degree hyperparameter and the error in the computed fit. In the bivariate case, we analyze the effect of precomposition with a scaling parameter $c$ as in Corollary \ref{cor:scal_rat}. For six variable functions, we examine the use of Algorithm \ref{alg:alt_fit} on data generated from tropical rational functions. Finally, we present preliminary experiments using the output of Algorithm \ref{alg:alt_fit} to initialize ReLU neural networks. All Matlab and Python codes to reproduce our experiments can be found at

\begin{center}
\url{https://github.com/Alex-Dunbar/Tropical-Data.git}.
\end{center}

\subsection{Univariate Data} We apply Algorithm \ref{alg:alt_fit} to a dataset consisting of 200 equally spaced points $x^{(i)} \in [-1,12]$ and corresponding $y$ values $y^{(i)} = \sin(x^{(i)}) + \epsilon^{(i)}$, where $\epsilon^{(i)} $ is independent zero-mean noise. Figure \ref{fig:univariate_fit} shows an example, with $d = 15$. We use a stopping criterion of $\eta^k \leq 10^{-12}$. The infinity norm of the error and the infinity norm of the update step at each iteration are plotted in Figure \ref{fig:sin_alternate_convergence}. Both the training loss and the update norm are nonincreasing and have regions on which they are constant.

\begin{figure}
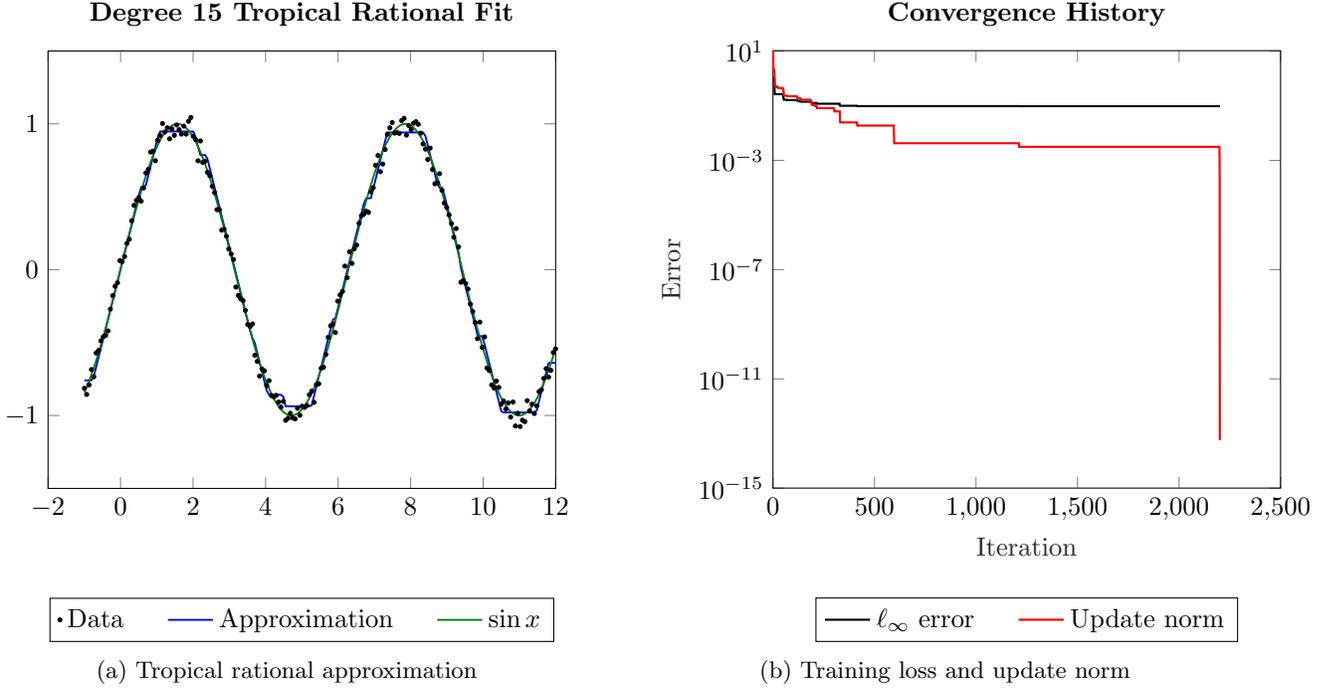

\centering
    \begin{subfigure}[b]{0.47\textwidth}
    \centering
    \input{Figures/Tikz/sin_1515_alternate_fit}
    \caption{Tropical rational approximation}
    \label{fig:sin_alternate_fit}
    \end{subfigure}
    \hfill
    \begin{subfigure}[b]{0.47 \textwidth}
    \centering
    \input{Figures/Tikz/sin_1515_alternate_convergence}
    \caption{Training loss and update norm}
    \label{fig:sin_alternate_convergence}
    \end{subfigure}
    \caption{Results of applying Algorithm \ref{alg:alt_fit} with degree 15 tropical rational functions to noisy data from a sine curve. Figure \ref{fig:sin_alternate_fit} shows the training data, the approximation by a tropical rational function, and the function $\sin x$. The approximating function captures the general behavior of the dataset. Figure \ref{fig:sin_alternate_convergence} shows the $\ell_\infty$ error $e^k = \|\mathbf{X}\boxplus \mathbf{p}^k - \mathbf{X}\boxplus \mathbf{q}^k - \mathbf{y}\|_\infty$ and the update norm $\eta^k = \|\begin{bmatrix}\mathbf{p}^{k+1} & \mathbf{q}^{k+1}\end{bmatrix}^\top - \begin{bmatrix}\mathbf{p}^{k} & \mathbf{q}^{k}\end{bmatrix}^\top\|_\infty$. Both the training loss and the update norm are nonincreasing and contain intervals on which they are nearly constant.} 
    \label{fig:univariate_fit}
\end{figure}

\paragraph{Effect of Degree} Here, we investigate the relationship between the degree of tropical rational function and the error in the fit. Specifically, we generate a dataset as in the above example and use Algorithm \ref{alg:alt_fit} to fit a tropical rational function of degree $d$ to the dataset for $d = 1,2,\ldots, 20$. As a stopping criterion in Algorithm \ref{alg:alt_fit}, we use $\eta^k \leq 10^{-12}$ or a maximum $k_{\max} = 10000$. Figure \ref{fig:degree_effect_sin} shows the relationship between the degree of the rational function and the error in the fit. Note that the error decreases as a function of the degree with a large decrease in error when the degree is 5. The number of iterations needed to achieve the stopping criterion is generally increasing but is not monotonic. 

\begin{figure}
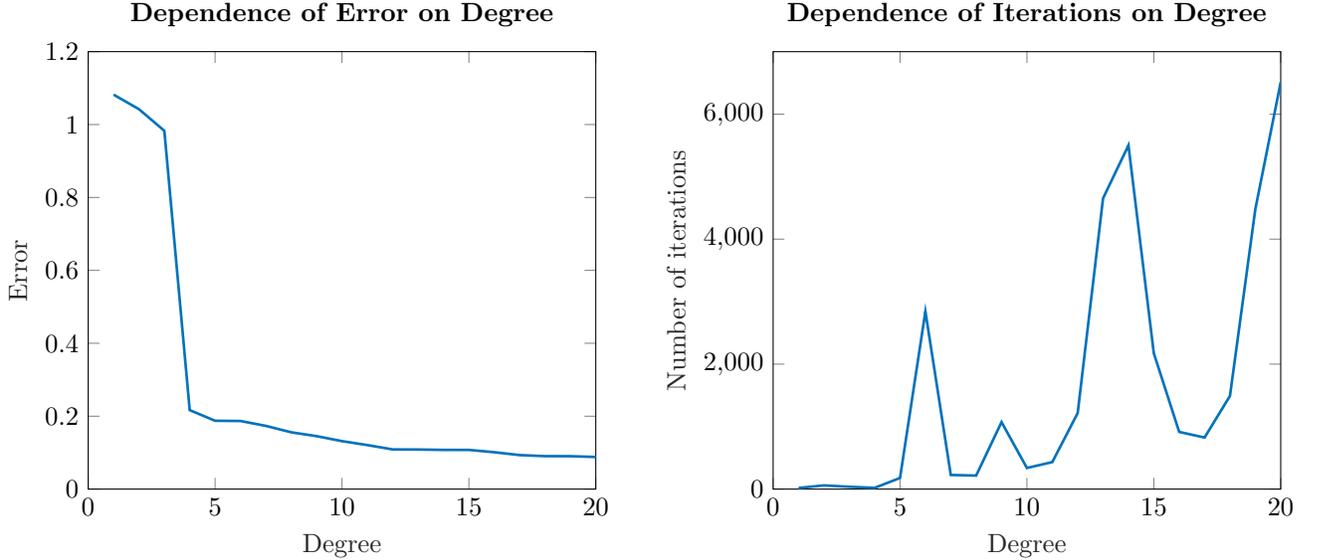

    \centering
    \begin{subfigure}[b]{0.47\textwidth}
        \input{Figures/Tikz/degree_error_sin}
    \end{subfigure}
    \hfill
    \begin{subfigure}[b]{0.47\textwidth}
        \input{Figures/Tikz/degree_iter_sin}
    \end{subfigure}
    \caption{Dependence of error and number of iterations on degree of tropical rational function fit to noisy data from a sine curve. The error decreases monotonoically as a function of degree with a large drop at degree 5. The number of iterations needed to reach the stopping criterion of $\eta^k \leq 10^{-12}$ generally increases with the degree.}
    \label{fig:degree_effect_sin}
\end{figure}

\subsection{Bivariate Data}

We use the method to approximate the Matlab \texttt{peaks} dataset using $\mathbf{d} = (10,10)$ and $\mathbf{d} = (31,31)$ and training until $\eta^k \leq 10^{-12}$. Explicitly, the \texttt{peaks} dataset consists of $2401 = 49^2$ equally spaced $(x_1,x_2)$ pairs in $[-3,3]^2$ and their evaluations 

\[\texttt{peaks}(x_1,x_2) = 3(1-x_1)^2e^{-x_1^2-(x_2+1)^2}-10\left(\frac{x_1}{5}-x_1^3-x_2^5\right)e^{-x_1^2-x_2^2}-\frac{1}{3}e^{-(x_1+1)^2 -x_2^2}.\] 

The fits and the error are shown below in Figure \ref{fig:Peaks_fit}. Note that in both cases there is error in the regions on which the data is nearly constant despite the piecewise linear nature of the tropical rational functions. As in the univariate case, the training error and the update norm are nonincreasing and have regions where they are constant over many iterations. 

\begin{figure}
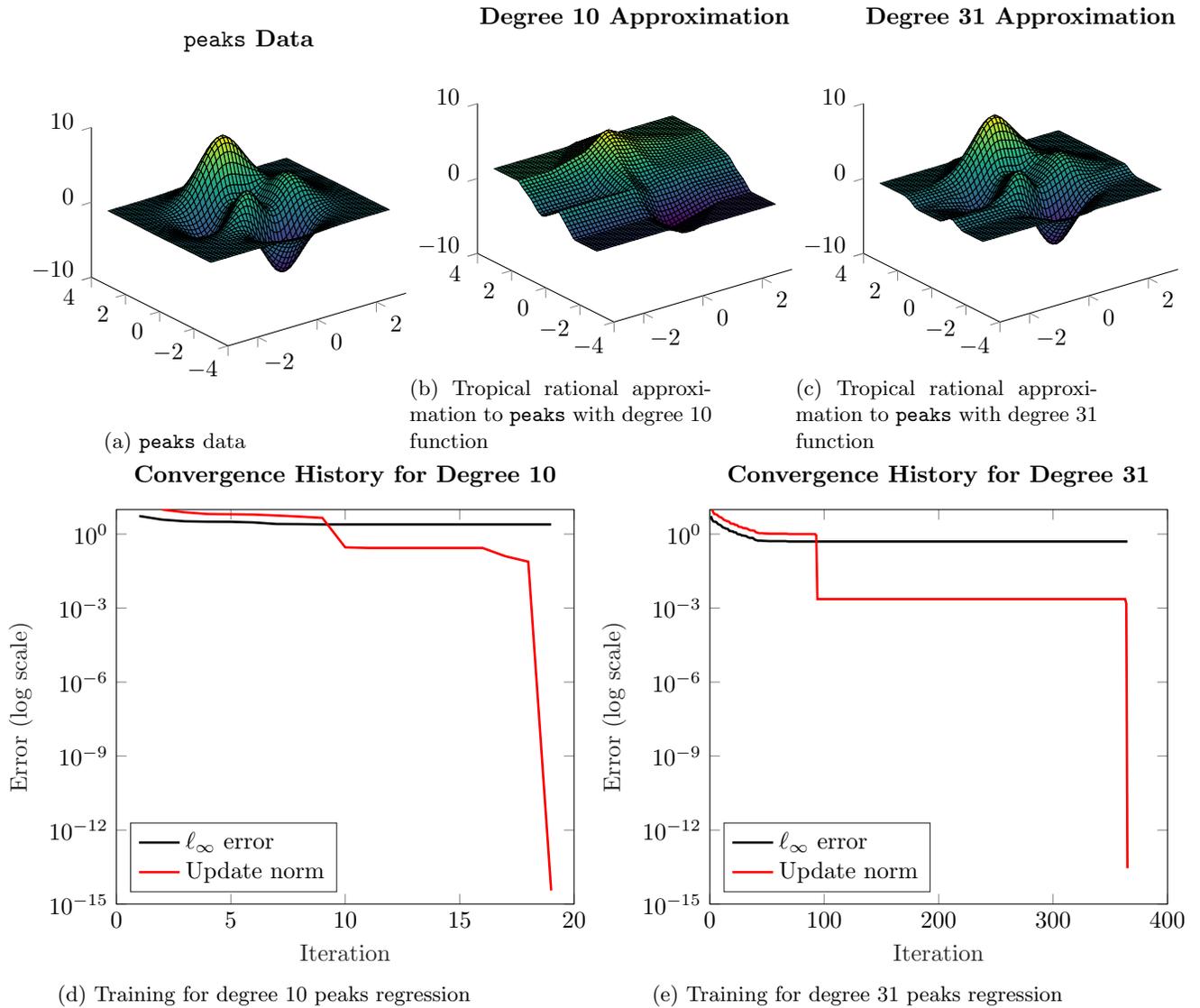

\centering
    \begin{subfigure}[b]{0.27\textwidth}
    \centering
    \input{Figures/Tikz/PeaksTrue}
    \caption{\texttt{peaks} data}
    \end{subfigure}
    \hspace{1cm}
    \begin{subfigure}[b]{0.27\textwidth}
    \input{Figures/Tikz/peaks_1010_alternate_fit}
    \caption{Tropical rational approximation to \texttt{peaks} with degree $10$ function}
    \label{fig:1010peaks}
    \end{subfigure}
    \hspace{1cm}
    \begin{subfigure}[b]{0.27\textwidth}
    \centering
    \input{Figures/Tikz/peaks_3131_alternate_fit}
    \caption{Tropical rational approximation to \texttt{peaks} with degree $31$ function}
    \label{fig:3131peaks}
    \end{subfigure}
    
    \begin{subfigure}[b]{0.47\textwidth}
    \centering
    \input{Figures/Tikz/peaks_1010_alternate_convergence}
    \caption{Training for degree 10 peaks regression}
    \label{fig:1010peakserror}
    \end{subfigure}
    \hfill
    \begin{subfigure}[b]{0.47\textwidth}
    \centering
    \input{Figures/Tikz/peaks_3131_alternate_convergence}
    \caption{Training for degree 31 peaks regression}
    \label{fig:3131peakserror}
    \end{subfigure}
    \caption{Results of applying Algorithm \ref{alg:alt_fit} with degree $10$ and $31$ tropical rational functions to the \texttt{peaks} dataset. The resulting degree $31$ function sketches the general behavior of the dataset (Figure \ref{fig:3131peaks}), while the degree $10$ function fails to approximate the data (Figure \ref{fig:1010peaks}). Figures \ref{fig:1010peakserror} and \ref{fig:3131peakserror} display the the $\ell_\infty$ error $e^k = \|\mathbf{X}\boxplus \mathbf{p}^k - \mathbf{X}\boxplus \mathbf{q}^k - \mathbf{y}\|_\infty$ and the update norm $\eta^k = \|\begin{bmatrix}\mathbf{p}^{k+1} & \mathbf{q}^{k+1}\end{bmatrix}^\top - \begin{bmatrix}\mathbf{p}^{k} & \mathbf{q}^{k}\end{bmatrix}^\top\|_\infty$. For both degrees, the training loss and the update norm are each nonincreasing and contain intervals on which they are nearly constant.}\label{fig:Peaks_fit}
\end{figure}

\paragraph{Effect of Scaling Parameter} In the above experiments, we directly fit a tropical rational function to the data. However, Corollary \ref{cor:scal_rat} suggests that we should fit a function of the form $f(c\mathbf{x})$, where $c \in \bbR$ and $f$ is a tropical rational function. To this end, we fit functions of the form $f(c\mathbf{x})$ for 21 equally spaced values of $c \in [1,3]$  and $f$ a tropical rational function of degree $35$. For each value of $c$, we use a stopping criterion of $\eta^k \leq 10^{-12}$ or a maximum of 500 iterations of the alternating method described in Algorithm \ref{alg:alt_fit} to find a tropical rational function $f$. The dependence of the training error on $c$ is shown in Figure \ref{fig:peaks_scaling_error} below. Note that the optimal value of $c$ is roughly 1.3. More generally, for fixed degree $d$, changing the value of $c$ gives a trade-off between maximum slope and resolution between slopes. 

\begin{figure}
    \centering
    \input{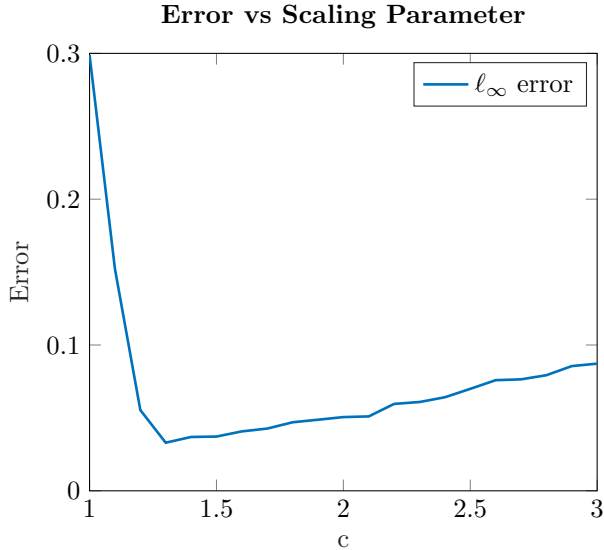}
    \caption{Error in approximation to the \texttt{peaks} dataset when using a degree $(35,35)$ tropical rational function with inputs scaled by $c$. Here, the optimal value of $c$ is roughly 1.3 and gives a much lower training error than the optimal function with unscaled inputs.}
    \label{fig:peaks_scaling_error}
\end{figure}

\subsection{Higher Dimensional Examples}

We test Algorithm \ref{alg:alt_fit} on functions with many variables. These experiments suggest that the alternating minimization method is able to find solutions with low training loss. However, these solutions do not appear to generalize well, even on data generated from tropical rational functions.

\paragraph{Regression on 6 Variable Function}

We fit a tropical rational function to the 6 variable function

\[g(\mathbf{x}) = x_1x_2x_3 + 2x_4x_5^2\sin(x_6^2)\]

\noindent on a training set consisting of $N = 10000$ points drawn uniformly at random from $[0,1]^6$ and then test on a test set generated in the same way. Here, we fix the maximum degree of the numerator and denominator to be 3 for each variable and train until $\eta^k \leq 10^{-12}$ or for a maximum of 500 iterations. There are 8192 trainable parameters. The convergence behavior during training is shown in Figure \ref{fig:6var_conv}. The $\ell^\infty$ error on the test set is $0.2721$, which is roughly $9.75$ times the final training error of $0.0279$.

\paragraph{Regression on 10 Variable Function}

We fit a tropical rational function to the 10 variable function

\[h(\mathbf{x}) = x_1x_2x_3 + 2x_4x_5^2\sin(x_6^2) - e^{x_7x_8x_9x_{10}}\]

\noindent on a training set consisting of $N = 10000$ points drawn uniformly at random from $[0,1]^{10}$ and then test on a test set generated in the same way. Here, we fix the maximum degree of the numerator and denominator to be 1 for each variable and train until $\eta^k \leq 10^{-12}$ or for a maximum of 500 iterations. There are 2048 trainable parameters. The convergence behavior during training is shown in Figure \ref{fig:10var_conv}. The $\ell^\infty$ error on the test set is $0.6828$, which is roughly $2.9$ times the final training error of $0.2342$.

\begin{figure}[t]
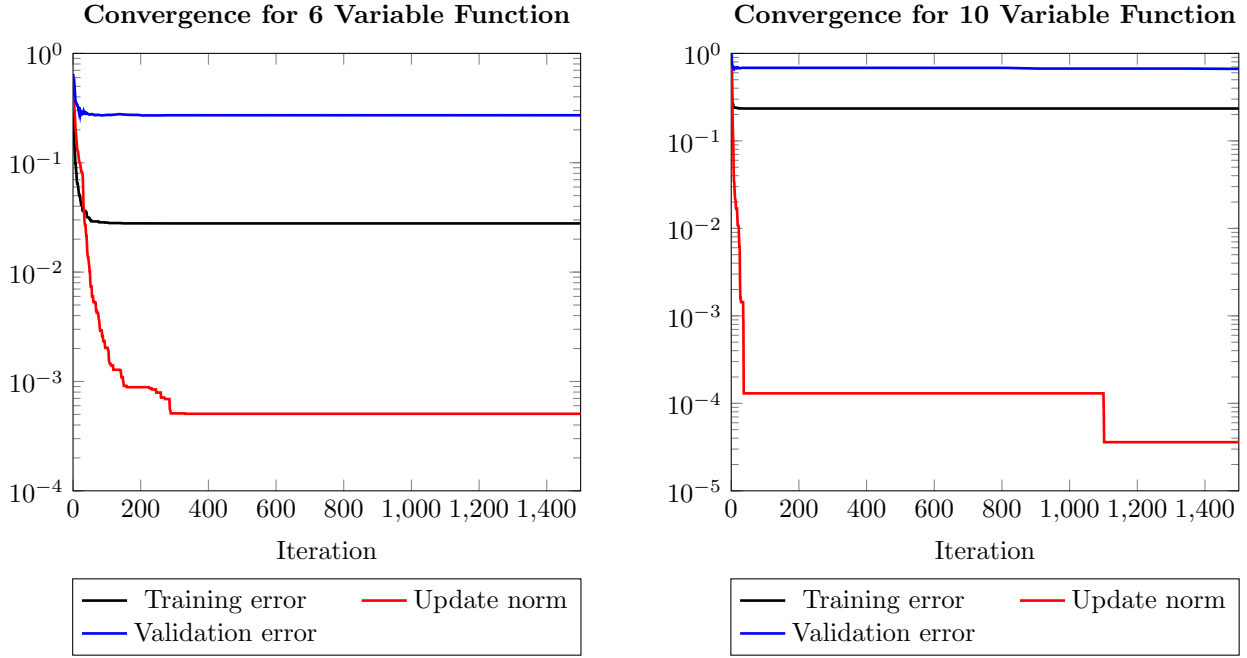

\centering
    \begin{subfigure}[b]{0.47\textwidth}
    \centering
    \input{Figures/Tikz/6var_alternate_convergence_with_test}
    \caption{Degree 3 fit for 6 variable function with 10,000 data points}
    \label{fig:6var_conv}
    \end{subfigure}
    \hfill 
    \begin{subfigure}[b]{0.47\textwidth}
    \centering
    \input{Figures/Tikz/10var_alternate_convergence_with_test}
    \caption{Degree 1 fit for 10 variable function and 10,000 data points}
    \label{fig:10var_conv}
    \end{subfigure}
    
    \caption{Convergence for tropical rational approximation of 6 and 10 variable functions. The training error and update norm display similar behavior as in the low dimensional cases with regions on which they remain constant.}
    
\end{figure}

\paragraph{Recovery of Tropical Rational Functions}

Here, we investigate the use of Algorithm \ref{alg:alt_fit} on data generated by tropical rational functions. Specifically, for $n  = 6$ we investigate the use of Algorithm \ref{alg:alt_fit} for the recovery of a tropical rational function of degrees 1 through 5 (i.e. $W_{d} = \{0,1,2,\ldots, d\}^6$ for $1\leq d\leq 5$). For each trial we generate a tropical rational function with coefficients sampled uniformly 
 at random from $[-5,5]$ as well as training and validation datasets of $N = 10000$ points sampled uniformly at random from $[-5,5]^6$. We then fit a tropical rational function $\hat{f}$ of the same degree using Algorithm \ref{alg:alt_fit} with a stopping criterion of $\eta^k\leq 10^{-8}$ or a maximum of 1000 iterations. In degrees at most 4, the method reached the stopping criterion in fewer than 1000 iterations for each trial. For degree 5, the method terminated after reaching 1000 iterations in 3 trials. In this experiment, $\mathbf{p}^0$ and $\mathbf{q}^0$ are initialized with entries drawn uniformly at random from $[-5,5]^6$. Table \ref{tab:6var_recover} shows the average relative training and validation loss $\|\hat{f}(\mathbf{x}) - \mathbf{y}\|_\infty / \|\mathbf{y}\|_\infty$ across the five trials in each degree. Here, the training loss is low, indicating that Algorithm \ref{alg:alt_fit} finds a near optimal solution. However, the validation loss is high and increasing as a function of the degree. 
 
\begin{table}
    \centering
    \begin{tabular}{c||c|c|c|c|c}
        Degree & 1&2&3&4&5 \\
        \hline
        Relative Training Error & 2.372 $\times 10^{-15}$ &
$5.869 \times 10^{-15}$ &
$9.108 \times 10^{-15}$ &
$1.286 \times 10^{-14}$ &
$9.373 \times 10^{-6} $\\
Relative Validation Error & 0.1271 &
0.2019 &
0.2869 &
0.3631 &
0.3598
    \end{tabular}
    \caption{Average training and validation error on data generated from 6 variable tropical rational functions. For each degree, the training loss is low, but the validation error is high and increasing as a function of degree.} 
    \label{tab:6var_recover}
\end{table}

\subsection{ReLU Neural Network Initialization}

Here we investigate the use of Algorithm \ref{alg:alt_fit} to initialize the weights of a ReLU neural network. In our experiments, we apply Algorithm \ref{alg:alt_fit} on data from the noisy sine curve and \texttt{peaks} datasets to generate approximations of the data then use the output tropical rational function to initialize the weights of ReLU networks. The architecture of the initialized network is determined by the number of monomomials in the tropical rational function $f$ used to initialize the network. 

The proof of \cite[Theorem 5.4]{zhang2018tropical} describes how to write a tropical rational function $f(\mathbf{x}) = p(\mathbf{x}) - q(\mathbf{x})$ as a ReLU neural network. If $g$ and $h$ are two tropical polynomials represented by neural networks $\nu$ and $\mu$, respectively, then 

\begin{equation}\label{eq:poly_sum_as_ReLU}
(g \oplus h)(\mathbf{x}) = \sigma((\nu-\mu)(\mathbf{x})) + \sigma(\mu(\mathbf{x})) - \sigma(-\mu(\mathbf{x}))) = \begin{bmatrix} 1 & 1 & -1 \end{bmatrix}\sigma\left(\begin{bmatrix}\nu(\mathbf{x}) - \mu(\mathbf{x})\\ \mu(\mathbf{x})\\ -\mu(\mathbf{x})\end{bmatrix}\right).
\end{equation}
In particular, the expression \eqref{eq:poly_sum_as_ReLU} can be applied to the case in which $g(\mathbf{x}) = \mathbf{w}^\top \mathbf{x} + g_{\mathbf{w}}$ is a tropical monomial. This allows us to take the maximum of two networks by adding a layer and appropriately concatenating weight matrices in the hidden layers. In the resulting architecture, each hidden layer decreases in width. For example a univariate degree $15$ tropical rational function $f$ can be represented via repeated applications of \eqref{eq:poly_sum_as_ReLU} as a neural network where the compositions are

\[\bbR^{1}\to \bbR^{48} \to \bbR^{24} \to \bbR^{12}\to \bbR^{6} \to \bbR^{1}.\]

For each dataset, we compare a network constructed as above to a fully connected ReLU network of the same architecture with weights initialized using the PyTorch default random weight initialization. All neural network parameter optimization is done in PyTorch using the Adam optimizer \cite{Adam} to minimize the MSE loss.

\subsubsection{Univariate Data}

We use a degree 15 tropical rational function to initialize a neural network to fit the noisy sin curve from above. The test data consists of 200 pairs $(x^{(i)},y^{(i)})$, where $x^{(i)}$ is randomly drawn points on the interval $[-1,12]$ and $y^{(i)} = \sin(x^{(i)})$. The networks are trained for 1000 epochs with batches of size 64 and a learning rate of $5\times 10^{-6}$ for the tropical initialized network and $10^{-3}$ for the randomly initialized network. We found choosing a smaller learning rate for the tropical initialization important to prevent the optimization from reducing the accuracy of the model. Training and validation errors are shown in Figure \ref{fig:sin15_nn}. The network initialized from a tropical rational function has lower training and validation error than the network with default initialization. 

\begin{figure}
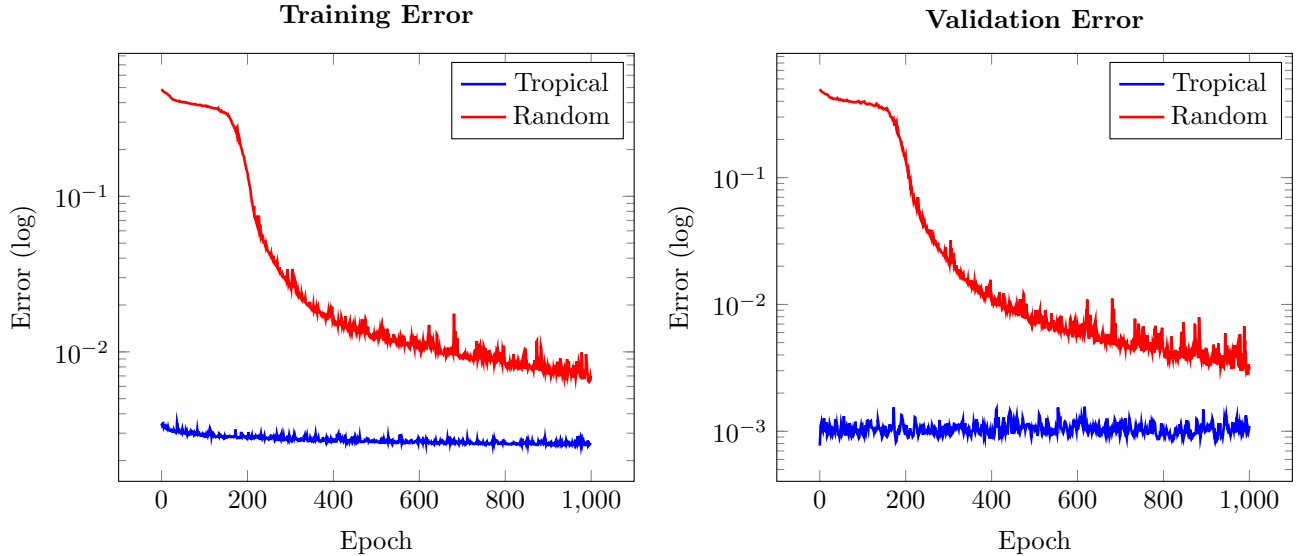

    \centering
    \begin{subfigure}[b]{0.47\textwidth}
    \input{Figures/Tikz/Sin15_train}
    \end{subfigure}
    \hfill
    \begin{subfigure}[b]{0.47\textwidth}
    \input{Figures/Tikz/Sin15_Test}
    \label{fig:my_label}
    \end{subfigure}
    \caption{Training and test loss for neural network fit to noisy $\sin$ data. The network initialized from a tropical rational approximation to the dataset starts and remains at lower training and validation losses than the network initialized with random weights.}\label{fig:sin15_nn}
\end{figure}

\subsubsection{Bivariate Data}

We use a degree $31$ tropical rational function to initialize the \texttt{peaks} dataset using Algorithm \ref{alg:alt_fit} as the initialization. The networks are trained for 100 epochs with a batch size of 64 and a learning rate of $10^{-4}$ for randomly initialized networks and $10^{-7}$ for the tropically initialized network. Results are shown in Figure \ref{fig:Peaks3131NN}. After roughly 20 epochs, the randomly initialized network outperforms the network initialized by a tropical rational function. 

\begin{figure}[t]
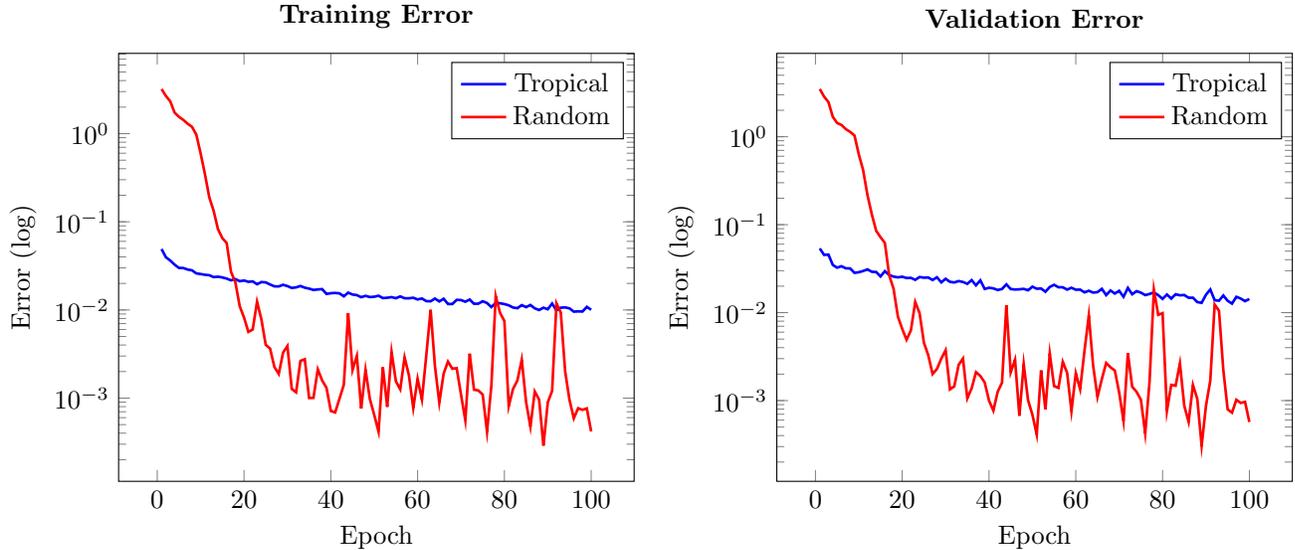

    \centering
    \begin{subfigure}[b]{0.47\textwidth}
    \input{Figures/Tikz/Peaks3131NN_train}
    \end{subfigure}
    \hfill
    \begin{subfigure}[b]{0.47\textwidth}
    \input{Figures/Tikz/Peaks3131NN_Test}
    \label{fig:Peaks3131NN_test}
    \end{subfigure}
    \caption{Training and test loss for neural network fits to \texttt{peaks} data. The randomly initialized network reaches lower training and validation errors than the tropically initialized network.}\label{fig:Peaks3131NN}
\end{figure}

\section{Conclusions}\label{sec:Conclusions}

We investigated the solution of regression with tropical rational functions by presenting an alternating heuristic. The proposed heuristic leverages known algebraic structure in tropical polynomial regression to iteratively fit numerator and denominator polynomials. Each iteration involves only (tropical) matrix-vector products and vector addition. The error at each iterate is nonincreasing, and each iterate is located in the nondifferentiability locus of the $\ell_\infty$ loss function. Computational experiments demonstrate that within a few iterations, our method can produce a qualitatively reasonable approximation of the input data. However, the optimal error and optimality conditions are unknown in general, preventing a quantitative evaluation of the heuristic. On datasets generated from tropical rational functions of low degrees where the true optimal error is known to be zero, the heuristic produces an approximation with very low training error. 

One potential application domain is in ReLU network initialization. In this work, we successfully initialized a ReLU network using a tropical rational function for a univariate regression task, while the tropical initialization was outperformed by random initialization for a bivariate regression task. This indicates the potential for future work to develop a better understanding of network initialization. In particular, the network architectures used in our experiments are limited and a full understanding of correspondences between network architectures and tropical functions and is currently an open problem.

Future work could help to develop a better theoretical understanding of the convergence behavior of Algorithm \ref{alg:alt_fit}. Additionally, future work could augment the polynomial regression steps using the ideas in \cite{HOOK2019maxplus2norm,tsiamis2019sparsity,tsilivis2022towardSparsityWeighted} to develop variants of Algorithm \ref{alg:alt_fit} for use with different norms or which enforce sparsity patterns or a regularization term. More generally, the development of a procedure for monomial selection remains open. 
\section{Acknowledgements}
This work was supported in part by NSF awards DMS 1751636, DMS 2038118, AFOSR grant FA9550-
20-1-0372, and US DOE Office of Advanced Scientific Computing Research Field Work Proposal 20-023231. 
\printbibliography

\end{document}